\documentclass[11pt]{amsart}
\usepackage{amssymb, latexsym, amsmath, amsfonts, hyperref, enumitem, tikz, fancyhdr}

\newtheorem{thm}{Theorem}[section]

\newtheorem{lem}[thm]{Lemma}

\theoremstyle{definition}
\newtheorem{defn}[thm]{Definition}
\theoremstyle{remark}

\numberwithin{equation}{section}

\hypersetup{
  colorlinks,
  citecolor=black,
  linkcolor=black,
  urlcolor=blue}
\theoremstyle{plain}

  {\end{enumerate}}

\hypersetup{
  colorlinks,
  citecolor=blue,
  linkcolor=blue,
  urlcolor=blue}

\begin{document}
\title{On Functional Calculus For $n$-Ritt Operators}
%{\let\thefootnote\relax\footnote{}}
%{\let\thefootnote\relax\footnote{The results in this paper are from the author's thesis ``The Carath\'{e}odory-Fej\'{e}r Interpolation Problems and the von Neumann Inequality" submitted to the Indian Institute of Science, Bangalore-560012.}}

%\thanks{The first named author is supported by Science and Engineering Research Board, DST, Government of India.}
\thanks{The second named author is supported by Council for Scientific and Industrial Research, MHRD, Government of India.}
%
%\author{Rajeev Gupta}
\author{Samya Kumar Ray}
%
%\address{Rajeev Gupta: Department of Mathematics and Statistics, Indian Institute of Technology, Kanpur-208016}
%\email{rajeevg@math.iitk.ac.in}

\address{Samya Kumar Ray: Department of Mathematics and Statistics, Indian Institute of Technology, Kanpur-208016}
\email{samya@math.iitk.ac.in}

\pagestyle{headings}

\begin{abstract}
In this paper, we present a new class of operators, which we name to be $n$-Ritt operators. This produces a discrete analogue of $n$-sectorial operators or multisectorial operators and generalizes the notion of Ritt operators and bi-Ritt operators. We develop a $H^\infty$ functional calculus for $n$-Ritt operators and prove an useful transference result. We also generalize the notions of quadratic functional calculus and $n$-$R$-Ritt operators and discuss some examples.
\end{abstract}
\maketitle
\section{Introduction}
The notion of functional calculus is quite old and has been developed and successfully used to several different fields of mathematics, e.g. operator theory, harmonic analysis, ordinary differential equations, partial differential equations and control theory. The main goal of functional calculus is to assign a well-defined meaning to $f(T),$ where $f$ is a function and $T$ is an operator (possibly unbounded). One of the famous result in this direction is the von Neumann inequality \cite{vN}, which is a non-trivial example of a bounded holomorphic functonal calculus. The von Neumann inequality was one of the main source of inspirations to develop various topics around functional calculus (see \cite{BC}). However, McIntosh and his coauthors (see \cite{A}, \cite{C}) developed $H^\infty$ functional calculus for sectorial operators and used it to study various problems in partial differential equations and harmonic anlysis. One of the major achievements of this development is the solution to Kato's square root poroblem. Unlike, the von Neumann inequality, the underlying Banach spaces, on which the operators act, can be anything. Very recently, functional calculus of sectorial operators has been effectively used to study maximal regularity problems for stochastic ordinary differential equations.

However, Christian Le Merdy (\cite{M}) developed functional calculus for Ritt operators and obtained many analogous results as one has in the case of sectorial operators (see \cite{AC} and \cite{M}). It has been noticed that the Ritt operators are discrete analogue of sectorial operators and they actually generate discrete analytic semigroups. Ritt operators have also been used to study discrete maximal regulaity problems (see \cite{K}, \cite{B}).

In this paper, we develop a notion of $n$-Ritt operators (one can also call it multi-Ritt operators.), which can be thought of a discrete analogue of $n$-sectorial (multisectorial) operators. It generalizes the notion of Ritt operators and bi-Ritt operators. Bisectorial operators and their notions of $H^\infty$ functional calculus have been successfully developed in \cite{VA}, \cite{A.K.} and \cite{BA}. Also, see \cite{MA} for a nice exposition on sectorial and Bisectorial operators and \cite{W} for an excellent survey on their applications. The notion of bi-Ritt operators was first introduced in. We define the notion of $n$-Ritt operators and the corresponding $H^\infty$ functional calculus and prove an important transference result, proved in \cite{M}. We also study the notion of $R$-$n$-Ritt operators and quadratic functional calculus and prove a corresponding transference result.

In Section \ref{Bi}, we recall the basics of $n$-sectorial operators. In Section \ref{RI}, we introduce the notion of $n$-Ritt operators and develop an $H^\infty$ functional calculus for this class of operators. In section \ref{TR}, we prove a transference result which relates the functional calculus of $n$-Ritt operators with the functional calculus of $n$-sectorial operators, which also can be used to get a good class of examples of $n$-Ritt operators. In section \ref{QU}, we introduce the notion of quadratic functional calculus and prove an analogous transference result. In section \ref{RI}, we develop a notion of $R$-$n$-Ritt operator and again prove a similar transference result.
\section{Functional Calculus for $n$-sectorial operators}\label{Bi}
We briefly recall neccessary preliminaries of $n$-sectorial operators and the useful notion of $R$-boundedness. We denote $\Omega_0:=\{+1,-1\}^{\mathbb{Z}}.$ The set $\Omega_0$ is a compact abelian group. Let us denote the normalized Haar measure on $\Omega_0$ by $\mathbb{P}.$ For any $k\in\mathbb{Z},$ define the $k$-th coordinate function $\epsilon_k$ by $\epsilon_k(\omega):=\omega_k,$ $\omega:=(\omega_k)_{k\in\mathbb{Z}}\in\Omega_0.$ The sequence of i.i.d. random variables $(\epsilon_k)_{k=1}^\infty$ is known to be the Rademacher variables. Let $X$ be a Banach space. We define the Banach space $\text{Rad}(X)$ to be the closure of $\text{span}\{\epsilon_k\otimes\omega_k:k\geq1, x_k\in X\}$ in the Bochner space $L^2(\Omega_0;X).$ Let $E\subseteq B(X)$ be a set of bounded operators in $X$. We say that the set $E$ is $R$-bounded provided there exists a constant $C>0,$ such that for any finite family $(T_k)_{k=1}^N$ of $E$ and a finite family $(x_k)_{k=1}^N$ in $X$, we have
\begin{equation}\label{R}
\big\|\sum_{k=1}^N\epsilon_k\otimes T_k(x_k)\big\|_{\text{Rad(X)}}\leq C\big\|\sum_{k=1}^N\epsilon_k\otimes x_k\big\|_{\text{Rad(X)}}.
\end{equation}
The smallest constant satisfying \eqref{R} is called the $R$-bound of $E$ and is denoted by $R_E$.

In the context of $n$-sectorial ($n$-Ritt) operators, we always fix a positive integer $n>1.$
\begin{defn}[$n$-sectorial ($R$-$n$-sectorial) operators]
For any $\omega\in(0,\frac{\pi}{n}),$ let us denote $\Sigma_\omega$ to be $\Sigma_{\omega}:=\{z\in\mathbb{C}^*:|argz|<\omega\},$ which is the open sector of an angle $2\omega$ around the positive real axis $(0,\infty)$. Let us define the set $\Sigma_{j,\omega}:=\exp\big({\frac{2ij\pi}{n}}\big)\Sigma_{\omega},$ where $0\leq j\leq n-1,$ and for any $T\subseteq\mathbb{C}$ and any complex number $a,$ we denote $aT=:\{az:z\in T\}.$ We define the open $n$-sector of an angle $\omega\in(0,\frac{\pi}{n}),$ as $S_{n,\omega}:=\cup_{j=0}^{n-1}\Sigma_{j,\omega}$ in $\mathbb{C}.$ Let $X$ be a Banach space. We say that a closed operator $A:D(A)\subseteq X\rightarrow X,$ with dense domain $D(A)$ is $n$-sectorial ($R$-$n$-sectorial) of type $\omega\in(0,\frac{\pi}{n})$ if and only if the following two conditions are satisfied.
\begin{enumerate}
\item[(1)]The spectrum of $A$ has the inclusion property $\sigma(A)\subseteq\overline{S}_{n,\omega}.$
\item[(2)]For any $\nu\in(\omega,\frac{\pi}{n}),$ the set $\{zR(z,A):z\in\mathbb{C}\setminus\overline{S}_{n,\nu}\}$ is bounded ($R$-bounded).
\end{enumerate}
\end{defn}

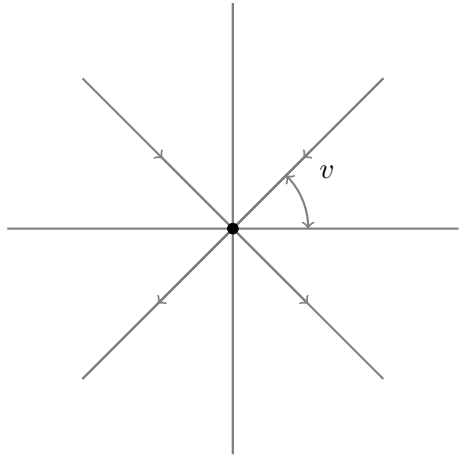
\begin{center}
\begin{tikzpicture}
\draw[gray, thick] (0,3)--(0,-3);
\draw[gray, thick] (0,-3)--(0,3);
\draw[gray, thick] (2,-2) -- (-2,2);
\draw[gray, thick] (-3,0) -- (3,0);
\draw[gray, thick] (-2,-2) -- (2,2);
\draw[gray, thick] (2,2) -- (-2,-2);
\draw[gray,thick] (0,0)--(45:1.5cm);
\draw[gray, thick,<->] (1,0) arc  (0:45:1);  
\node at (1.25,.75){$v$};
\draw[gray,thick,-<] (0,0)--(1,1);
\draw[gray,thick,-<] (0,0)--(-1,1);
\draw[gray,thick,->] (0,0)--(-1,-1);
\draw[gray,thick,->] (0,0)--(1,-1);
%pic["$\alpha$", draw=orange, <->, angle eccentricity=1.2, angle radius=1cm] {angle=(1,0)--(0,0)--(1,1)};
\filldraw[black] (0,0) circle (2pt) node[anchor=west]{};
\end{tikzpicture}
\end{center}

For any $\theta\in(0,\frac{\pi}{n})$, let $H_0^\infty(S_{n,\theta})$ denote the algebra of all bounded holomorphic functions $f:S_{n,\theta}\to\mathbb{C},$ for which there exists $c,s>0$ depending only on $f,$ such that 
\begin{equation}
|f(z)|\leq c\frac{|z|^s}{1+|z|^{2s}},\ \ \text{for all}\ z\in S_{n,\theta}.
\end{equation}It is easy to see that $H_0^\infty(S_{n,\theta})$ equipped with the supremum norm is a Banach algebra. Let $0<\omega<\theta<\frac{\pi}{n}$ and $f\in H_0^\infty(S_{n,\theta})$, we define
\begin{equation}\label{C}
f(A):=\frac{1}{2\pi i}\int_{\Gamma_{S_{n,\nu}}}f(z)R(z,A),
\end{equation}
where $\nu\in(\omega,\theta)$ and the contour in \eqref{C} is defined as follows. First, we define the contour $\Gamma_{\nu,n}:=\{te^{i\nu}:\infty>t>0\}\oplus\{te^{-i\nu}:0<t<\infty\},$ which is $\partial\Sigma_{\nu}$ oriented anticlockwise. Thereafter, define $\Gamma_\nu^j:=\exp\big({\frac{2ij\pi}{n}}\big)\Gamma_{\nu,n},$ where $0\leq j\leq n-1.$ It is clear that each $\Gamma_\nu^j$ is $\Sigma_{j,\nu}$ oriented anticlockwise. Therefore, the image of the contour defined by $\oplus_{j=0}^{n-1}\Gamma_\nu^j$ is the boundary of $S_{n,\nu}$ oriented anticlockwise. We denote $\Gamma_{S_{n,\nu}}$ by $\oplus_{j=0}^{n-1}\Gamma_\nu^j$. The integral, defined in \eqref{C} is an improper one. It converges absolutely due to the adequate decay of the resolvent operator. One can use Cauchy's theorem to see that $f(A)$ is well defined.
%$$\partial S_\nu:=\{te^{i\nu}:\infty>t>0\}\oplus\{-te^{-i\nu}:0<t<\infty\}\oplus\{-te^{i\nu}:0<t<\infty\}\oplus\{te^{-i\nu}:0<t<\infty\}$$ and $\nu\in(\omega,\theta)$. The decay on the resolvent operator ensures that the integral in \eqref{C} is absolutely convergent and by Cauchy's theorem it is independent of $\nu$. Therefore $f(A)$ is well-defined.
%\begin{center}
%\begin{tikzpicture}
%\draw[gray, thick,->] (0,3)--(0,-3);
%\draw[gray, thick,->] (0,-3)--(0,3);
%\draw[gray, thick,->] (2,-2) -- (-2,2);
%\draw[gray, thick,->] (-3,0) -- (3,0);
%\draw[gray, thick,->] (-2,-2) -- (2,2);
%\draw[gray, thick,->] (2,2) -- (-2,-2);
%\draw[gray, thick,->] (1,0) arc  (0:45:1);  
%%pic["$\alpha$", draw=orange, <->, angle eccentricity=1.2, angle radius=1cm] {angle=(1,0)--(0,0)--(1,1)};
%\filldraw[black] (0,0) circle (2pt) node[anchor=west]{};
%\end{tikzpicture}
%\end{center}
%
%\begin{center}
%\begin{tikzpicture}
%\draw[gray, thick] (0,3)--(0,-3);
%\draw[gray, thick] (0,-3)--(0,3);
%\draw[gray, thick,->] (1,0) -- (-1,2);
%\draw[gray,thick,->] (1,0)-- (-1,-2);
%\draw[gray, thick] (-4,0) -- (4,0);
%\draw[gray, thick] (1,0) arc  (0:360:1);
%\draw[gray, thick] (3,0) arc  (0:360:1);
%\filldraw[black] (0,0) circle (2pt) node[anchor=west]{};
%\filldraw[black] (2,0) circle (2pt);
%\end{tikzpicture}
%\end{center}

One can prove that the map $\Phi:H_0^\infty(S_{n,\theta})\to B(X),$ defined as $\Phi(A):=f(A),$ is an algebra homomorphism (see \cite{DA} for Bisectorial operators). We say that $A$ admits a bounded $H^\infty(S_{n.\theta})$ functional calculus if and only if the algebra homomorphism $\Phi$ is bounded, i.e. one has \begin{equation*}
\|f(A)\|_{X\to X}\leq C\|f\|_{\infty,S_{n,\theta}},
\end{equation*} 
for some positive constant $C>0,$ which is independent of $f$ and $f\in H_0^\infty(S_{n,\theta}).$
\section{Functional Calculus for $n$-Ritt operators}\label{RI}In this section, we introduce the notion of $n$-Ritt operators and develop an appropriate notion of $H^\infty$ functional calculus.
\begin{defn}[$n$-Ritt ($R$-$n$-Ritt) operator]
For any, $\gamma\in(0,\frac{\pi}{n})$, let $\mathcal{B}_\gamma$ be the interior of the convex hull of $1$ and the disc $D(0,\sin\gamma)$ in the complex plane. Then, $\mathcal{B}_\gamma$ is called the Stolz domain of angle $\gamma\in(0,\frac{\pi}{n})$. Let us denote $\Delta_\gamma:=1-\mathcal{B}_\gamma$ and $\Delta_{j,\gamma}:=\exp(\frac{2ij\pi}{n})\Delta_\gamma.$ We define the open $n$-Stolz domain of an angle $\gamma\in(0,\frac{\pi}{n})$ as  $\mathbb{B}_{n,\gamma}:=1-\big(\cup_{j=0}^{n-1}\Delta_{j,\gamma}\big).$
An operator $T:X\to X$ is said to be an $n$-Ritt ($R$-$n$-Ritt) operator of type $\alpha\in(0,\frac{\pi}{n})$ if it satisfies the following conditions. 
\begin{enumerate}
\item[(1)]The spectrum of $T$ entertains the inclusion property $\sigma(T)\subseteq\overline{\mathbb{B}}_{n,\gamma}.$
\item[(2)]For any $\beta\in(\alpha,\frac{\pi}{n}),$ the set $\{(\lambda-1)R(\lambda,T):\lambda\in\mathbb{C}\setminus\overline{\mathbb{B}}_{n,\beta}\}$ is bounded ($R$-bounded).
\end{enumerate}
\end{defn}

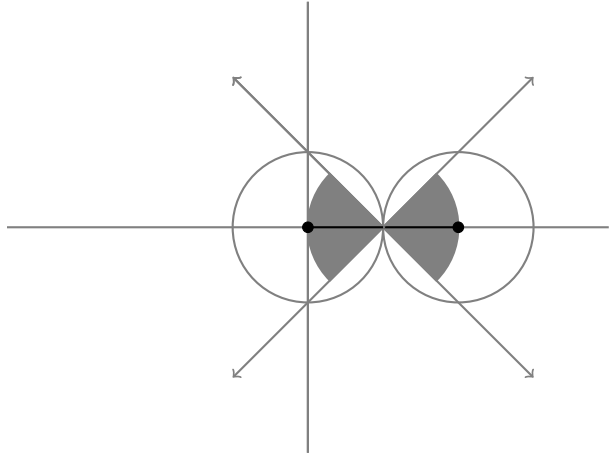
\begin{center}
\begin{tikzpicture}
\draw[gray, thick] (0,3)--(0,-3);
\draw[gray, thick] (0,-3)--(0,3);
\draw[gray, thick,->] (1,0) -- (-1,2);
\draw[gray,thick,->] (1,0)-- (-1,-2);
\draw[gray, thick,->] (1,0) -- (-1,2);
\draw[gray,thick,->] (1,0)-- (3,2);
\draw[gray,thick,->] (1,0)-- (3,-2);
\draw[gray, thick] (-4,0) -- (4,0);
\draw[gray, thick] (1,0) arc  (0:360:1);
\draw[gray, thick] (3,0) arc  (0:360:1);
\draw[gray, thick,fill=gray] (1,0) --  (2,0) arc  (0:45:1) -- cycle;
\draw[gray, thick,fill=gray] (1,0) -- (2,0) arc  (0:-45:1) -- cycle;
\draw[black,thick] (1,0)--(2,0);
\draw[gray, thick,fill=gray] (1,0) --  (0,0) arc  (180:225:1) -- cycle;
\draw[gray, thick,fill=gray] (1,0) --  (0,0) arc  (180:135:1) -- cycle;
\draw[black,thick] (1,0)--(0,0);
\filldraw[black] (0,0) circle (2pt) node[anchor=west]{};
\filldraw[black] (2,0) circle (2pt);
\end{tikzpicture}
\end{center}

The following lemma reveals the intimate connection of $n$-Ritt and $n$-sectorial operator.
\begin{lem}An operator $T:X\to X$ is an $n$-Ritt ($R$-$n$-Ritt) operator, then the operator $A=I-T$ is an $n$-sectorial ($R$-$n$-sectorial) operator.
%\begin{enumerate}
%\item[(i)]$\sigma(T)\subseteq D(0,1)\cup\{1\}\cup D(2,1)$ and
%\item[(ii)]$A=I-T$ is $n$-sectorial ($R$-$n$-Sectorial). 
%\end{enumerate}
\end{lem}
For any, $\gamma\in(0,\frac{\pi}{n})$, we let $H_0^\infty(\mathbb{B}_{n,\gamma})$ to be the space of all bounded holomorphic functions $\phi:\mathbb{B}_{n,\gamma}\to\mathbb{C},$ such that 
\begin{equation}
|\phi(\lambda)|\leq c|1-\lambda|^s,\ \text{for}\ \lambda\in\mathbb{B}_{n,\gamma},
\end{equation}where the constants $c>0$ and $s>0$ depend only on $\phi.$ It is easy to see that the set $H_0^\infty(\mathbb{B}_{n,\gamma}),$ equipped with the supremum norm is a Banach algebra.
Suppose $T$ is an $n$-Ritt operator of type $\alpha\in(0,\frac{\pi}{n})$ and $\gamma\in(\alpha,\frac{\pi}{n})$. Then, for any $\phi\in H_0^\infty(\mathbb{B}_{n,\gamma})$, let us define
\begin{equation}\label{ppp}
\phi(T):=\int_{\Gamma_{\mathbb{B}_{n,\beta}}}\phi(\lambda)R(\lambda,T)d\lambda,
\end{equation}
where, $\beta\in(\alpha,\gamma)$ and $\Gamma_{\mathbb{B}_{n,\beta}}$ is the boundary $\partial\mathbb{B}_{n,\beta},$ oriented counterclockwise. Define the contours 
\begin{eqnarray}
\Gamma_{\mathcal{B}_{\beta}}^1 &=& 1+[\sin\beta\exp(i(\frac{\pi}{2}-\beta))-1],\ 0\leq t\leq 1\\
\Gamma_{\mathcal{B}_{\beta}}^2&=&\sin\beta\exp\theta,\ \frac{\pi}{2}-\beta\leq\theta\leq\frac{3\pi}{2}+\beta\\
\Gamma_{\mathcal{B}_{\beta}}^3&=&\sin\beta\exp(i(\frac{3\pi}{2}+\beta))+t[1-\sin\beta\exp(i\frac{3\pi}{2})],\ 0\leq t\leq 1.
\end{eqnarray}Then, the contour defined by $\Gamma_{\mathcal{B}_{\beta}}:=\Gamma_{\mathcal{B}_{\beta}}^1\oplus\Gamma_{\mathcal{B}_{\beta}}^2\oplus\Gamma_{\mathcal{B}_{\beta}}^3
$ is the boundary of $\mathcal{B}_\beta$ oriented counterclockwise. Define the contour $\Gamma_{\Delta_\beta}:=1-\Gamma_{\mathcal{B}_\beta},$ which is the boundary of $\Delta_\beta$ oriented counter clockwise. Similarly, the contour defined by $\Gamma_{\Delta_{j,\beta}}:=\exp(\frac{2ij\pi}{n})\Gamma_{\Delta_\beta}$ is the boundary of $\Delta_{j,\beta}$ oriented anticlockwise.  we define the contour $\Gamma_{\mathbb{B}_{n,\beta}}:=1-\oplus_{j=1}^{n-1}\Gamma_{\Delta_{j,\beta}}.$ The assumptions on the decay of the resolvent of $T$ confirms that the integral in \eqref{ppp} is absolutely convergent and it does not depend on the angle $\beta$ by Cauchy's theorem. 

We state the following theorem without proof. The proof is a routine check.
\begin{thm}The map $\phi\mapsto\phi(T)$ is an algebra homomorphism from $ H_0^\infty(\mathbb{B}_{n,\gamma})$ to $B(X)$.
\end{thm}
\begin{lem} Let $T$ be an $n$-Ritt operator of type $\alpha\in(0,\frac{\pi}{n})$, then $rT$ is an $n$-Ritt operator for any $r\in(0,1)$ and
\begin{itemize}
\item[(1)] For any $\beta\in(\alpha,\frac{\pi}{n})$, the set $\{(\lambda-1)R(\lambda,rT):r\in(0,1),\lambda\in\mathbb{C}\backslash\mathbb{B}_{n,\beta}\}$ 
is bounded.
\item[(2)]For any $\gamma\in(\alpha,\frac{\pi}{n}),$ and $\phi\in H_0^\infty(\mathbb{B}_{n,\gamma})$, $\phi(T)=\lim_{r\to 1^-}\phi(rT)$.
\end{itemize}
%\begin{proof}
%\begin{itemize}
%\item[(1)]First consider, $\beta\in(\alpha,\frac{\pi}{n})$. Clearly, if $\lambda\in\mathbb{C}\backslash\mathbb{B}_{n,\beta}$ and $r\in(0,1)$, then $\frac{\lambda}{r}\in\mathbb{C}\backslash\mathbb{B}_{n,\beta}$. Otherwise one will have that
%\begin{equation*}
%\lambda\in\mathbb{C}\backslash\mathbb{B}_{n,\beta}\ \text{and}\ \lambda\in r\overline{\mathbb{B}_{n,\beta}}=r\overline{\text{CONV}D(0,\sin\gamma)\cup\{1\}\cup D(0,\sin\gamma)}
%\end{equation*}which will be a contradiction. Hence, we have that $\lambda\notin\sigma(rT)$ and as a consequence
%\begin{equation*}
%(\lambda-1)R(\lambda,rT)=\frac{\lambda-1}{\lambda-r}(\frac{\lambda}{r}-1)R(\frac{\lambda}{r},T).
%\end{equation*}
%Now, as the set $\{(\lambda-1)(\lambda-r)^{-1}:r\in(0,1),\lambda\in\mathbb{C}\backslash\mathbb{B}_\beta\}$ is bounded, we get the required conclusion.
%\item[(2)]This follows easily by using (1) and applying Lebesgue's dominated convergence theorem. 
%\end{itemize}
%\end{proof}
\end{lem}
\begin{proof}
\begin{itemize}
\item[1.]
Let us consider $\beta\in(\alpha,\frac{\pi}{n})$ and $r\in(0,1).$ It is easy to see that if $\lambda\in\mathbb{C}\setminus\mathbb{B}_{n,\beta},$ then $\frac{\lambda}{r}\in\mathbb{C}\setminus
\overline{\mathbb{B}}_{n,\beta}$ and $\lambda\not\in\sigma(rT).$ Therefore, we have $$(\lambda-1)R(\lambda,rT)=\frac{\lambda-1}{\lambda-r}\big(\frac{\lambda}{r}-1\big)R\big(\frac{\lambda}{r},T\big).$$
Now, as the set $\{(\lambda-1)(\lambda-r)^{-1}:r\in(0,1),\lambda\in\mathbb{C}\backslash\mathbb{B}_\beta\}$ is bounded, we get the required conclusion.
\item[2.] This follows easily by using (1) and applying Lebesgue's dominated convergence theorem. 
\end{itemize}
\end{proof}

We now use the above lemma to extend the functional calculus of $n$-Ritt operators to a larger class of holomorphic functions. Let $H_{c,0}^\infty(\mathbb{B}_{n,\gamma})\subseteq H^\infty(\mathbb{B}_{n,\gamma})$
be the linear span of $H_0^\infty(\mathbb{B}_{n,\gamma})$ and the constant functions. It is immediate that $H_{c,0}^\infty(\mathbb{B}_{n,\gamma})$ is a unital Banach algebra equipped with the supremum norm. For
any $\psi=c+\phi,\ \phi\in H_0^\infty(\mathbb{B}_{n,\gamma}),$ let us define $\psi(T):=cI_X+\phi(T).$ Therefore, we have a unital algebra homomorphism  $\psi\mapsto\psi(T)$ from $H_{c,0}^\infty(\mathbb{B}_{n,\gamma})$ 
to $B(X).$ Let $\mathbb{C}[Z]$ be the set of all complex polynomials. Suppose $\phi\in\mathbb{C}[Z].$ One can represent $\phi$ as $\phi(z)=(z-1)\tilde{\phi}(z)+\phi(1),$ where $\tilde{\phi}\in\mathbb{C}[z]$ is a 
polynomial. Therefore, we conclude that $\phi\in H_{c,0}^\infty(\mathbb{B}_{n,\gamma}).$ Thus, one can use Runge's theorem to see that $H_{c,0}^\infty(\mathbb{B}_{n,\gamma})$ contains rational functions with poles off $\overline{\mathbb{B}}_{n,\gamma}.$ Therefore, $H_{c,0}^\infty(\mathbb{B}_{n,\gamma})$ contains all
polynomials. Therefore, for any $T$ as above and any $r\in(0,1),$ we have $\sigma(rT)=r\sigma(T)\subseteq\mathbb{B}_{n,\beta}.$ Hence, the definition of $\phi(rT)$ given by agrees with the usual Dunford-Riesz functional
calculus and by the above lemma for any rational function $\phi$ with poles off $\overline{\mathbb{B}}_{n,\gamma},$ the above definition of $\phi(T)$ coincides with the one obtained by substituting $T$ in $\phi.$ 
Thus this notion of the functional calculus agrees with the usual functional calculus for polynomials.
\begin{defn}
Let $T$ be an $n$-Ritt operator of type $\alpha\in(0,\frac{\pi}{n})$ and $\gamma\in(\alpha,\frac{\pi}{n})$. We say that $T$ admits a bounded $H^\infty(\mathbb{B}_{n,\gamma})$ functional calculus if there exists a 
constant $K>0$, such that
\begin{equation*}
\|\phi(T)\|\leq K\|\phi\|_{\infty,\mathbb{B}_{n,\gamma}},\ \text{for all}\ \phi\in H_0^\infty(\mathbb{B}_{n,\gamma}).
\end{equation*}
\end{defn}
The following theorem shows that it is enough to show the boundedness of the functional caculus on the set of polynomials.
\begin{thm}Let $T$ be a $n$-Ritt operator of type $\alpha\in(0,\frac{\pi}{n})$ and $\gamma\in(\alpha,\frac{\pi}{n})$. We say that $T$ admits a bounded $H^\infty(\mathbb{B}_{n,\gamma})$ functional calculus if and only if 
\begin{equation}\label{n}
\|\phi(T)\|\leq K\|\phi\|_{\infty,\mathbb{B}_{n,\gamma}},
\end{equation}
for all polynomial $\phi\in\mathbb{C}[Z].$
\end{thm}
\begin{proof}
One direction is trivial. To prove the reverse direction. let us assume that \ref{n} holds for all polynomials $\phi\in\mathbb{C}[Z].$ We fix a $\phi\in H_0^\infty(\mathbb{B}_{n,\gamma}).$ Let $r\in(0,1)$ and 
$r^\prime\in(r,1).$ Suppose $\Gamma$ be the boundary of $r^\prime\mathbb{B}_{n,\gamma}$ oriented counterclockwise.

By Runge's theorem there exists a sequence of polynomials $\{\phi_n\}_{n\in\mathbb{N}}$ such that $\phi_m\to\phi$ uniformly on compact subsets of $r^\prime\mathbb{B}_{n,\gamma}.$ Since $\sigma(rT)\subseteq 
r^\prime\mathbb{B}_{n,\gamma},$ we have that $$\phi_m(rT)=\frac{1}{2\pi i}\int_\Gamma\phi_m(\lambda)R(\lambda,rT)d\lambda\to\phi(rT)$$
as $m\to\infty.$ Thus we have
\begin{equation}
\|\phi_m(rT)\|\leq K\|\phi_m\|_{\infty,r\mathbb{B}_{n,\gamma}}\leq K\|\phi_m\|_{\infty,r^\prime\mathbb{B}_{n,\gamma}},\ \text{for all}\ m\in\mathbb{N}.
\end{equation}
Taking $m\to\infty$ and then $r\to 1,$ we obtain that $$\|\phi(T)\|\leq K\|\phi\|_{\infty,\mathbb{B}_{n,\gamma}}.$$ This completes the proof of the lemma.
\end{proof}
\section{A transfer principle}\label{TR}In this section, we prove a transfer principle between functional calculus of $n$-Ritt operators and $n$-sectorial operators. We can obtain various examples and 
counterexamples with the help of this transference principle. Nevertheless, at the end of this section, we construct some concrete examples of $n$-Ritt operators.
\begin{thm}\label{T}
Let $T:X\to X$ be an $n$-Ritt operator. Denote $A=I-T$. Then, the following are equivalent:
\begin{itemize}
\item[(1)]$T$ admits a bounded $H^\infty(\mathbb{B}_{n,\gamma})$ functional calculus for some $\gamma\in(0,\frac{\pi}{n})$
\item[(2)]$A$ admits a bounded $H^\infty(S_{n,\theta})$ functional calculus for some $\theta\in(0,\frac{\pi}{n}).$
\end{itemize}
\begin{proof}
Denote $\tilde{\Delta}_{n,\gamma}=1-\mathbb{B}_{n,\gamma}$, for $\gamma\in(0,\frac{\pi}{n})$.

To see, how (i) implies (ii), we fix a $f\in H_0^\infty(S_{n,\gamma})$ and define the following map $$\phi(\lambda)=f(1-\lambda).$$ It is immediate that, $\phi:\mathbb{B}_{n,\gamma}\to\mathbb{C}$ is holomorphic. 
Again, we observe the following $$|\phi(\lambda)|=|f(1-\lambda)|\leq c|1-\lambda|^s,$$ where $\lambda\in \mathbb{B}_{n,\gamma}$, and $c,\ s>0$ are some positive real numbers. Therefore, we conclude that
$\phi\in H_0^\infty(\mathbb{B}_{n,\gamma})$. Also, one can notice that $\|\phi\|_{\infty,\mathbb{B}_{n,\gamma}}\leq\|\phi\|_{\infty,1-S_{n,\gamma}}=\|f\|_{\infty,S_{n,\gamma}}$. Therefore, we have the following 
observation
\begin{align*}
f(A)&=\frac{1}{2\pi i}\int_{\partial S_{n,\beta}}f(z)(zI-A)^{-1}dz\\
&=\frac{1}{2\pi i}\int_{1-\partial S_{n,\beta}}f(1-z)R(z,T)dz\\
&=\frac{1}{2\pi i}\int_{1-\partial S_{n,\beta}}\phi(\lambda)R(\lambda,T)d\lambda\\
&=\frac{1}{2\pi i}\int_{\partial\mathbb{B}_{n,\beta}}\phi(\lambda)R(\lambda,T)d\lambda\\
&=\phi(T),
\end{align*}where in the second last step, we have made use of the Cauchy's theorem. 
Hence, we obtain the following estimate $$\|\phi(T)\|=\|f(A)\|\leq K\|\phi\|_{\infty,\mathbb{B}_{n,\gamma}}\leq K\|f\|_{\infty,S_{n,\gamma}}.$$  We shall turn our attention to prove the implication (ii) to (i).

Suppose $A$ admits a bounded $H_\infty(S_{n,\theta})$ functional calculus, 
for some $\theta\in(0,\frac{\pi}{n}).$ Then, we clearly have $\sigma(A)\subset\overline{\tilde{\Delta}}_{n,\alpha}$, for some $\alpha\in(0,\frac{\pi}{n})$. Taking $\theta$ close enough to $\frac{\pi}{n}$, 
we may assume $\alpha<\theta$. We fix $\gamma\in(\theta,\frac{\pi}{n})$ and choose $\beta\in(\theta,\gamma)$. Let $\phi\in H_0^\infty(\mathbb{B}_{n,\gamma})$ and $f:\tilde{\Delta}_{n,\gamma}\to\mathbb{C}$ 
is defined as $f(z)=\phi(1-z).$ Clearly, $\phi$ is a holomorphis function, which satisfies $\|\phi\|_{\infty,\mathbb{B}_{n,\gamma}}=\|f\|_{\infty,\tilde{\Delta}_{n,\gamma}}.$ We also have the estimate
$|f(z)|\leq c|z|^s$, where $z\in\tilde{\Delta}_{n,\gamma}$ and $c,s$ are some positive constants depending only on $\phi.$ Let us define two contours, which will be necessary for the rest of the proof. Consider the contour
$\Gamma_1:=\oplus_{j=0}^{n-1}\exp\big(\frac{2ij\pi}{n}\big)\big(1-\big(\Gamma_{\mathcal{B}_{\beta}}^1\oplus \Gamma_{\mathcal{B}_{\beta}}^3\big)\big)$ and 
$\Gamma_2:=\oplus_{j=0}^{n-1}\exp\big(\frac{2ij\pi}{n}\big)\big(1-\big( \Gamma_{\mathcal{B}_{\beta}}^2\big).$
Next, we define the functions, 
$f_1:\mathbb{C}\backslash\Gamma_1\to\mathbb{C}$ and $f_2:\mathbb{C}\backslash\Gamma_2\to\mathbb{C}$ as the following
\begin{equation}
f_1{(z)}=\frac{1}{2\pi i}\int_{\Gamma_1}\frac{f(\lambda)}{\lambda-z}d\lambda\ \ \text{and}\ f_2{(z)}=\frac{1}{2\pi i}\int_{\Gamma_2}\frac{f(\lambda)}{\lambda-z}d\lambda.
\end{equation}

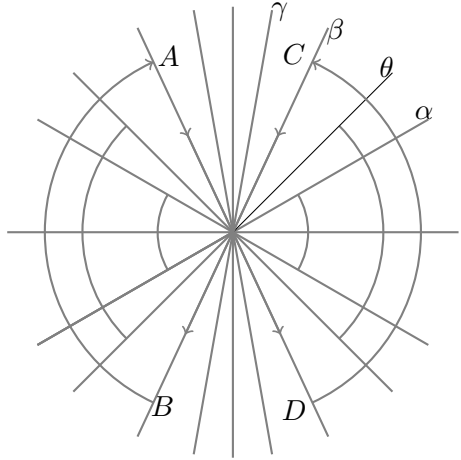
\begin{center}
\begin{tikzpicture}
\draw[gray, thick,] (0,3)--(0,-3);
\draw[gray, thick,] (-3,0)--(3,0);
\node at (78:3cm) {$\gamma$};
\node at (63:3cm) {$\beta$};
\node at (71:2.5cm) {$C$};
\node at (47:3cm) {$\theta$};
\node at (32:3cm) {$\alpha$};
\node at (-71:2.5cm) {$D$};
\node at (248:2.5cm) {$B$};
\node at (110:2.5cm) {$A$};
\draw[gray,thick] (0,0)-- (210:3cm) ;
\draw[gray,thick] (1,0) arc(0:30:1);\draw[gray,thick] (1,0) arc(0:-30:1);
\draw[gray,thick] (-1,0) arc(180:150:1);\draw[gray,thick] (-1,0) arc(180:210:1);
\draw[gray,thick] (2,0) arc(0:45:2);\draw[gray,thick] (2,0) arc(0:-45:2);
\draw[gray,thick] (-2,0) arc(180:135:2);\draw[gray,thick] (-2,0) arc(180:225:2);
\draw[gray,thick] (0,0)-- (30:3cm) ;
\draw[gray,thick] (0,0)-- (210:3cm) ;
\draw (0,0)--(45:3cm);
\draw[gray,thick] (0,0)--(225:3cm);
\draw[gray,thick] (0,0)--(210:3cm);
\draw[gray,thick] (0,0)--(65:3cm);
\draw[gray,thick,-<] (0,0)--(65:1.5cm);
\draw[gray,thick,-<] (0,0)--(115:1.5cm);
\draw[gray,thick,->] (0,0)--(-65:1.5cm);
\draw[gray,thick,->] (0,0)--(-115:1.5cm);
\draw[gray,thick] (0,0)--(245:3cm);
\draw[gray,thick] (0,0)--(80:3cm);
\draw[gray, thick,->] (2.5,0) arc  (0:65:2.5);
\draw[gray, thick] (2.5,0) arc  (0:-65:2.5);
\draw[gray, thick,->] (-2.5,0) arc  (180:115:2.5);
\draw[gray, thick] (-2.5,0) arc  (180:245:2.5);
\draw[gray,thick] (0,0)--(-30:3cm);
\draw[gray,thick] (0,0)--(-45:3cm);
\draw[gray,thick] (0,0)--(-65:3cm);
\draw[gray,thick] (0,0)--(-80:3cm);
\draw[gray,thick] (0,0)--(150:3cm);
\draw[gray,thick] (0,0)--(135:3cm);
\draw[gray,thick] (0,0)--(115:3cm);
\draw[gray,thick] (0,0)--(100:3cm);
\draw[gray,thick] (0,0)--(260:3cm);
\end{tikzpicture}
\end{center}

Clearly, we have $f(z)=f_1(z)+f_2(z)$, for $z\in\tilde{\Delta}_{n,\beta}$. Since, the distance between $\Gamma_1$ and $S_\theta\setminus\tilde{\Delta}_{n,\theta}$ is strictly positive, and 
$\Gamma_1\subseteq\tilde{\Delta}_{n,\gamma}$, there exists $c_1>0,$ such that
\begin{equation}
|f_1(z)|\leq c_1\|f\|_{\infty,\tilde{\Delta}_{n,\gamma}}\ \text{for all}\ z\ \text{in}\ S_{n,\theta}\backslash\tilde{\Delta}_{n,\theta}.
\end{equation}
Also, distance from $\Gamma_2$ from $\tilde{\Delta}_{n,\theta}$ is strictly positive. Thus, we have $|f_2(z)|\leq c_2\|f\|_{\infty,\tilde{\Delta}_{n,\gamma}}$, for all $z$ in $\tilde{\Delta}_{n,\theta}$.
Hence, we conclude that $\|f_1\|_{\infty,S_{n,\theta}}\leq c\|f\|_{\infty,\tilde{\Delta}_{n,\gamma}}$. Let $g:S_{n,\theta}\to\mathbb{C}$ be defined by
\begin{equation}
g(z)=f_1(z)+\frac{af_2(0)}{a+z},
\end{equation}where the complex number $a$ is chosen large enough such that $a\in\rho(T),$ where $\rho(T)$ is the resolvent of $T.$
Clearly, $zf_1(z)$ is bounded as $|z|\to\infty$ and so $zg(z)$ is bounded on $S_{n,\theta}$. Also, one has that $|f_2(z)-f_2(0)|\leq c_3|z|$ on $\tilde{\Delta}_{n,\theta}$. So, we have for $z$ in 
$\tilde{\Delta}_{n,\theta}$,
\begin{align*}
g(z)&=f(z)+\left(\frac{af_2(0)}{a+z}-f_2(z) \right)\\
&=f(z)+(f_2(0)-f_2(z))-f_2(0)\frac{z}{a+z},
\end{align*}
Hence, we have that $|g(z)\leq c_4\max\{|z|^s,|z|\}$, for $z$ in $\tilde{\Delta}_{n,\theta}$. Hence, $g\in H_0^\infty(S_{n,\theta})$ and we have that $f_1(A)=g(A)-f_2(0)(aI+A)^{-1}$. Also, it is evident that 
$\|f_1(A)\|\leq c_5\|f_1\|_{\infty,\Sigma_{n,\theta}}\leq C_5\|f\|_{\infty,\tilde{\Delta}_{n,\gamma}}$. Now, as we have 
\begin{equation}
f_2(A)=\frac{1}{2\pi i}\int_{\Gamma_2}f(\lambda)R(\lambda,A)d\lambda,
\end{equation}
we immediately see that $\|f_2(A)\|\leq c_6\|f\|_{\infty,\tilde{\Delta}_{n,\gamma}}$.
Since, $\phi(T)=f_1(A)+f_2(A)$, we get that $\|\phi(T)\|\leq C\|\phi\|_{\infty,\mathbb{B}_{n,\gamma}}.$
\end{proof}
\end{thm}

Now, we turn our attention to some concrete examples. For this we recall the basics of Schauder multipliers. For more on Schauder multipliers and associated notions, we recommend \cite{JF}.

\begin{defn}[Schauder multipliers] Let $X$ be a Banach space and $(e_m)_{m\geq 0}$ be a Schauder basis of $X.$ For a complex sequence $(\gamma_m)_{m\geq 0}$ the operator defined by 
$$D(T_\gamma):=\{x=\sum_{m\geq 1}x_me_m:\sum_{m\geq 0}\gamma_mx_me_m \ \text{exists}\}$$ and $T_\gamma x:=\sum_{m\geq 0}\gamma_mx_me_m,\ x\in D(T)$ is called the Schauder multiplier associated to the sequence
$(\gamma_m)_{m\geq 0}.$
\end{defn}
Let $BV$ denote the set of all complex sequences $\gamma=(\gamma_m)_{m\geq 0},$ such that the total variation of $\gamma$, denoted by $\|\gamma\|_{BV}:=\sum_{m=1}^\infty|\gamma_m-\gamma_{m-1}|$ is finite. It 
is well known that the set $(BV,\|.\|_{BV})$ is a Banach space and $BV\subseteq\ell^{\infty}(\mathbb{N}_0).$ Given any Banach space $X$ with a Schauder basis $(e_m)_{m\geq 0},$ and a complex sequence
$(\gamma_m)_{m\geq 0},$ it is not hard to prove that the associated Schauder multiplier $T_\gamma$ is a bounded operator and $\|T_\gamma\|_{X\to X}\leq K\|\gamma\|_{BV},$ for some constant $K>0.$ 
\begin{thm}\label{CONS}Let $X$ be a Banach space with a Schauder basis $(e_m)_{m\geq 0}$. Let $\gamma=(\gamma_n)_{n\geq 0}\in BV$ be an increasing sequence of positive real numbers such that we have 
\[\lim_{n\to\infty}\gamma_n=1.\] Then, the associated Schauder multiplier $T_\gamma$ is a Ritt operator. 
\end{thm}
\begin{proof}
Let us fix $\theta\in(-\pi,0)\cup(0,\pi].$ We define the sequence $$\gamma(\theta)_n=\frac{1}{e^{i\theta}-\gamma_n}, n\geq 0.$$ We observe the following 
\begin{eqnarray*}
\sum_{n=1}^\infty|\gamma(\theta)_n-\gamma(\theta)_{n-1}|&=&\sum_{n=1}^\infty\big|\int_{\gamma_{n-1}}^{\gamma_n}\frac{dt}{(e^{i\theta}-t)^2}\big| \\
&\leq &\sum_{n=1}^\infty\int_{\gamma_{n-1}}^{\gamma_n}\frac{dt}{|e^{i\theta}-t|^2}\\
&=&\int_{0}^{1}\frac{dt}{|e^{i\theta}-t|^2}.\\
\end{eqnarray*}
We denote $I_\theta:=\int_{0}^{1}\frac{dt}{|e^{i\theta}-t|^2}.$ Since the integral is finite, we conclude that the sequence defined by $\gamma(\theta):=(\gamma(\theta)_n)_{n\geq 0},$ is in $BV$ for all  
$\theta\in(-\pi,0)\cup(0,\pi].$ One can easily verify that the operator $e^{i\theta}-T_\gamma$ is invertible and $R(e^{i\theta},T_\gamma)$ is a Schauder multiplier associated to the sequence $\gamma(\theta).$
An elementary computation yields that $$I_\theta=\frac{\pi-\theta}{2\sin\theta},\ \theta\in(0,\pi).$$ Therefore, we notice that $$|e^{i\theta}-1|I_\theta=\frac{\pi-\theta}{2\cos\frac{\theta}{2}},\ \theta\in(0,\pi).$$ 
Also $I_\theta=I_{-\theta},\ \theta\in(0,\pi).$ Thus it follows that the quantity $\sup\{\|(\lambda-1)R(\lambda,T_\gamma)\|:|\lambda|=1,\lambda\neq 1\}$ is finite and that 
$\sigma(T_\gamma)\subseteq\mathbb{D}\cup\{1\}.$ Now applying maximal principle to the function $(z-1)R(z,T_\gamma),$ for $|z|>1,$ we deduce that the operator $T_\gamma$ is a Ritt operator.
\end{proof}We recall the following interpolation theorem due to Carleson.
\begin{thm}[Carleson's interpolation theorem]\label{cal}Let $(z_i)_{i\geq 0}$ be a sequence in $\Sigma_{\frac{\pi}{2}},$ then the following conditions are equivalent:
\begin{itemize}
\item[1.] There is a $\delta>0$ such that for all $j\geq 0,$ we have $$\prod_{i\neq j}\big|\frac{z_i-z_j}{z_i+z_j}\big|\geq\delta.$$
\item[2.]There exists a sequence $(f_i)_{i\geq 0}$ in $H^\infty(\Sigma_{\frac{\pi}{2}})$ and a constant $M>0$ such that for all $i\geq 0,$ $f_i(z_i)=1,$ for all $j\neq i,$ $f_i(z_j)=0$ and 
$$\sup_{z\in H^\infty(\Sigma_{\frac{\pi}{2}})}\sum_{i\geq 0}|f_i(z)|\leq M.$$ 
\end{itemize}
\end{thm}
Such a sequence in $\Sigma_{\frac{\pi}{2}}$ is called an interpolating sequence for $H^\infty(\Sigma_{\frac{\pi}{2}}).$

Let $X$ be a Banach space and $1\leq p\leq\infty,$ we define a Banach space $X_p:=X\oplus_p X,$ with the norm as $\|(x_1,x_2)\|_p=\big(\|x_1\|^p+\|x_2\|^p\big)^{\frac{1}{p}}.$ It is easy to check that if
$T\in B(X)$ is a Ritt operator of type $\alpha\in(0,\frac{\pi}{2}),$ then the operator $S$ defined as $S(x_1,x_2):=(Tx_1,(2I_X-T)x_2)$ is a Bi-Ritt operator of type $\alpha$ and $T$ admits a 
bounded $H^\infty$ functional calculus if and only if $S$ admits a bounded $H^\infty$ functional calculus. Take a Schauder basis $(e_m)_{m\geq 0}$ of $X.$ If $(e_m)_{m\geq 0}$ is an unconditional one,
then Schauder multiplier as in the Theorem \ref{CONS} admits a bounded $H^\infty$ functional calculus. However, if $(e_m)_{m\geq 0}$ is not a Schauder basis, one can use the Theorem \ref{cal} to construct a 
Schauder multiplier which does not admit a bounded $H^\infty$ functional calculus as follows. Take a sequence $\gamma=(\gamma_n)_{n\geq 0}$ and consider the associate Schauder multiplier $T_\gamma.$ If the 
sequence $(1-\gamma_n)_{n\geq 0}$ is an interpolating sequence and $I-T_\gamma$ admits a bounded functional calculus, then by Theorem \ref{cal}, every bounded sequence becomes a Schauder multiplier, which 
is a contradiction.
\section{Quadratic functional calculus}\label{QU}
In this section, we define the quadratic functional calculus for $n$-Ritt operators and prove a transfer result as in the Theorem (\ref{TR}). 
           
Let $g_1,\dots,g_n$ be a finite family of $H^\infty(\Omega)$, where $\Omega$ is a non empty open subset of $\mathbb{C}$. Let us define the following norm
\begin{equation}
\big\|(\sum_{j=1}^n|g_j|^2)^{\frac{1}{2}}\big\|_{\infty,\Omega}=\sup_{z\in\Omega}\big(\sum_{j=1}^n|g_j(z)|^2\big)^{\frac{1}{2}}.
\end{equation}

\begin{defn}[Quadratic functional calculus for $n$-sectorial operators] Let $A$ be a $n$-sectorial operator of type $\omega\in(0,\frac{\pi}{n})$ on a Banach space $X$, and $\theta\in(\omega,\frac{\pi}{n})$.
We say that $A$ admits a quadratic $H^\infty(S_{n,\theta})$ functional calculus, if there exists a constant $C>0$, such that for any $n\geq 1$, and for any $g_1,\dots, g_n\in H_0^\infty(S_{n,\theta})$, and 
for any $x\in X$ we have
\begin{equation}
\big\|\sum_{k=1}^n\epsilon_k\otimes g_k(A)(x_k)\big\|_{\text{Rad(X)}}\leq C\|x\|\big\|(\sum_{k=1}^n|g_k|^2)^{\frac{1}{2}}\big\|_{\infty,\Omega}.
\end{equation}
\end{defn}
\begin{defn}[Quadratic functional calculus for $n$-Ritt operators] Let $T$ be a Bi-Ritt operator of type $\alpha\in(0,\frac{\pi}{n})$ and $\gamma\in(\alpha,\frac{\pi}{n})$. We say that $T$ admits a quadratic 
$H^\infty(\mathbb{B}_{n,\gamma})$ functional calculus if there exists a constant $K>0$, such that for any $n\geq 1$, and for any $\phi_1,\dots,\phi_n\in H^\infty(\mathbb{B}_{n,\gamma})$, and for any $x\in X$, 
we have
\begin{equation}
\big\|\sum_{k=1}^n\epsilon_k\otimes\phi_k(T)(x_k)\big\|_{\text{Rad(X)}}\leq C\|x\|\big\|(\sum_{k=1}^n|g_k|^2)^{\frac{1}{2}}\big\|_{\infty,\mathbb{B}_{n,\gamma}}.
\end{equation}
\begin{thm}  Let $T:X\to X$ be an $n$-Ritt operator. Denote $A=I-T$. Then, the following are equivalent.
\begin{itemize}
\item[(1)]$T$ admits a quadratic $H^\infty(\mathbb{B}_{n,\gamma})$ functional calculus for some $\gamma\in(0,\frac{\pi}{n})$.
\item[(2)]$A$ admits a quadratic $H^\infty(S_{n,\theta})$ functional calculus for some $\theta\in(0,\frac{\pi}{n}).$
\end{itemize}
\begin{proof}
The proof follows exactly as in theorem \ref{T}.
\end{proof}
\end{thm}
\end{defn}
\textbf{Acknowledegement:} The author expresses his sincere gratitude to his thesis advisor Prof. Parasar Mohanty for many stimulating discussions. He is indebt to Prof. Christian Le Merdy for many valuable
comments and insight.
%\begin{center}
%\begin{tikzpicture}
%\draw[gray, thick] (0,3)--(0,-3);
%\draw[gray, thick] (-2,2) -- (2,-2);
%\draw[gray, thick,] (-3,0) -- (3,0);
%\draw[gray, thick] (-2,-2) -- (2,2);
%\filldraw[black] (0,0) circle (2pt) node[anchor=west]{};
%\end{tikzpicture}
%\end{center}

\end{document}